\documentclass{amsart}

\usepackage{amsmath}
\usepackage{amssymb}
\usepackage{amsthm}
\usepackage{color}
\usepackage{graphicx}
\usepackage{psfrag}
\usepackage{pstool}

\makeatletter
\g@addto@macro{\endabstract}{\@setabstract}
\newcommand{\authorfootnotes}{\renewcommand\thefootnote{\@fnsymbol\c@footnote}}%
\makeatother

\newcommand{\beq}{\begin{equation}}
\newcommand{\beqa}{\begin{eqnarray}}
\newcommand{\bx}{\bar{x}^\star}

\newcommand{\eeq}{\end{equation}}
\newcommand{\eeqa}{\end{eqnarray}}
\newcommand{\geqs}{\geqslant}

\newcommand{\leqs}{\leqslant}

\newcommand{\Rb}{{\mathbb R}}
\newcommand{\sgn}{\operatorname{sgn}} 

\newcommand{\xdownarrow}[1]{{\left\downarrow\vbox to #1{}\right.\kern-\nulldelimiterspace}}
\newtheorem{lemma}{Lemma}
\newtheorem{prop}{Proposition}
\newtheorem{conj}{Conjecture}

\theoremstyle{definition}
\theoremstyle{remark}

\begin{document}

\title[]{Bifurcating solutions in a non-homogeneous boundary value problem for a nonlinear pendulum equation.}

\author[F.P. da Costa]{Fernando P. da Costa}
\address{Departamento de Ci\^encias e Tecnologia, Universidade Aberta, Lisboa, Portugal, and
Centre for Mathematical Analysis, Geometry, and Dynamical Systems, Instituto Superior T\'e\-cni\-co,
Universidade de Lisboa, Lisboa, Portugal}
\email{fcosta@uab.pt}


\author[M. Grinfeld]{Michael Grinfeld}
\address{Department of Mathematics and Statistics, University of Strathclyde, 
	Glasgow, United Kingdom}
\email{m.grinfeld@strath.ac.uk}


\author[J.T. Pinto]{Jo\~ao T. Pinto}
\address{Departamento de Matem\'atica, Instituto Superior T\'ecnico, Universidade de Lisboa, Lisboa, Portugal, and
	Centre for Mathematical Analysis, Geometry, and Dynamical Systems, Instituto Superior T\'e\-cni\-co,
	Universidade de Lisboa, Lisboa, Portugal}
\email{jpinto@math.tecnico.ulisboa.pt}


\author[K. Xayxanadasy]{Kedtysack Xayxanadasy}
\address{Department of Mathematics, Faculty of Natural Sciences, National University of Laos, 
Dongdok Campus, Vientiane, Laos}
\email{kedtysack@gmail.com}


\date{July 29, 2019}



\maketitle



\begin{abstract}
	Motivated by recent studies of bifurcations in liquid crystals cells \cite{cggp,cmp}
	we consider a nonlinear pendulum ordinary differential equation in the bounded interval $(-L, L)$ 
	with non-homogeneous mixed boundary conditions (Dirichlet an one end of the 
	interval, Neumann at the other) and study the bifurcation diagram of its solutions 
	having as bifurcation parameter the size of the interval, $2L$,
	and using techniques from phase space analysis,
	time maps, and asymptotic estimation of integrals, 
	complemented by appropriate numerical evidence. 
\end{abstract}



\section{Introduction}\label{sec:intro}

Motivated by the study of the twist-Fr\'eedericksz transition in 
a nematic liquid crystal cell, a non-homogeneous Dirichlet 
boundary value problem for the nonlinear pendulum equation
\begin{equation}
x''(t) + \sin 2x(t) = 0, \label{pendulum}
\end{equation}
for $t$ in the interval $[-L, L],$ was considered in recent papers \cite{cggp,cmp}, 
and the structure of the bifurcating solutions when the parameter $L$ is 
changed was studied.

\medskip

In this paper we consider again the existence of solutions for a non-homogeneous boundary 
value problem for equation~\eqref{pendulum}, this time 
with a Dirichlet condition at $t=-L$ and a Neumann one at $t=L,$
a case that may be relevant for modelling the twist-Fr\'eedericksz transition in
a cholesteric liquid crystal cell \cite{nm}. The problem that will be considered
is the following, illustrated in Figure~\ref{fig1},

\begin{eqnarray}
& & 
\begin{cases}
x'=y\\
y'=-\sin 2x,
\end{cases} \label{originalsystem} \\
& & 
\rule{0mm}{5mm}x(-L)=-\phi,\qquad y(L)=\phi^\star.
\label{originalbc}
\end{eqnarray}

\noindent 
where $\phi^\star := \sqrt{1-\cos 2\phi}.$ 
%
%
%
\begin{figure}[h]
	\psfragfig*[scale=0.60]{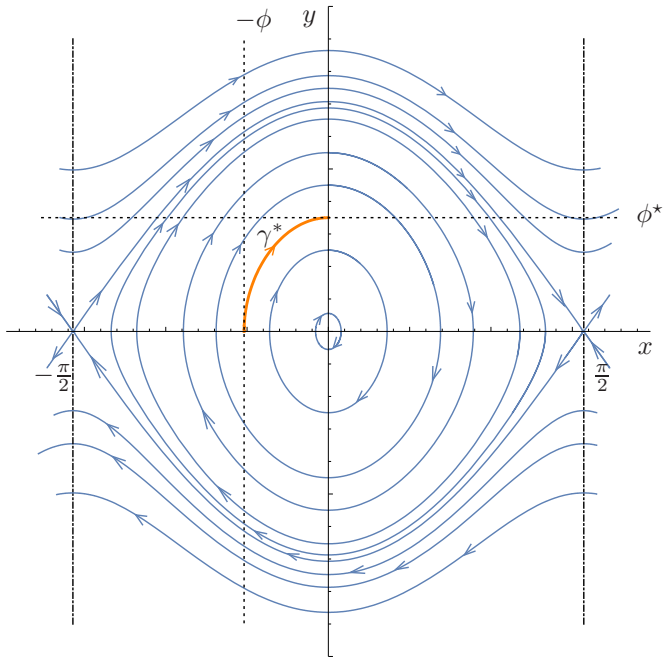}{%
	\psfrag{x}{$x$}
	\psfrag{y}{$y$}
	\psfrag{fi}{$-\phi$}
	\psfrag{fi*}{$\phi^\star$}
	\psfrag{g}{$\gamma^*$}
	\psfrag{pi/2}{$\frac{\pi}{2}$}
	\psfrag{-pi/2}{$-\frac{\pi}{2}$}
}
	\caption{Phase plot of the orbits of equation~\eqref{originalsystem}, the boundary conditions
		\eqref{originalbc} to be considered, and the orbit $\gamma^*$ referred to in the text. 
		The straight lines $x= -\frac{\pi}{2}$ and $x=\frac{\pi}{2}$ are to be identified.}\label{fig1}
\end{figure}
%
%
%
System \eqref{originalsystem} has a first integral given by 
\begin{equation}
V(x,y) = y^2 -\cos 2x. \label{V}
\end{equation}
and from $V(-\phi,0)=V(0,\phi^\star)$ we conclude that the points $(-\phi, 0)$ 
and $(0,\phi^\star)$ lie on the
same orbit of \eqref{originalsystem}, as illustrated in Figure~\ref{fig1} by the orbit denoted by $\gamma*$. 
This relation between the conditions imposed at the two boundary points entails a certain symmetry in the 
allowed solutions, akin to what happened in the case of Dirichlet boundary conditions studied in \cite{cggp}, 
and is a reasonable first step towards the understanding of the general case. 

\medskip

The tools used in this paper are based on appropriately defined time maps, measuring the
time spent by a given orbit between two of its points. According to what will be most
appropriate for the computations, we will identify an orbit by
the ordinate of its first intersection either with the $x$-axis, the $y$-axis, or the
line $x=-\phi$, leading to different, although equivalent, time maps.

\medskip

We study the bifurcation diagram of solutions to \eqref{originalsystem}--\eqref{originalbc} 
using the following procedure: we start by identifying a segment of an orbit
of \eqref{originalsystem}, $\gamma^*$, such that the corresponding solution, in
addition to satisfying the boundary conditions \eqref{originalbc}, $x(-L)=-\phi$ and
$y(L)=\phi^*$, also satisfies $y(-L)=0$ and $x(L) = 0$ (see Figure~\ref{fig1}).
We call the solution corresponding to $\gamma^*$ a \emph{critical solution}
(and $\gamma^*$ a \emph{critical (segment of an) orbit}).
Calling
this solution \emph{critical} is justified as we will prove that in
bifurcation diagrams parameterized by $L$, there is more than one
solution branch passing through it.
To this (segment of) orbit $\gamma^*$ corresponds a \emph{critical time} $T^*$, 
and a corresponding critical value of
$L=L^* = T^*/2>0.$ 

We then perturb this (segment of) orbit and investigate 
how the time spent changes relative to $T^*$. This time is
measured by adequately defined time maps, whose definition arises naturally
from the phase portrait and the first integral \eqref{V} (see, e.g. \cite{korman,smoller}).
This approach was 
used in \cite{cggp,cmp} for the study of \eqref{originalsystem}
with non-homogeneous Dirichlet boundary conditions; 
its application in the present case led to
some unexpected difficulties and the analytical study had to be completed with numerical
simulations providing solid evidence for a conjecture about the existence of
a single minimum of the time maps of some solution branches.

\section{Time maps: definition and basic results}\label{sec:timemaps}

For every $\alpha\in\left(0, \frac{\pi}{2}\right)$, the
orbit $\gamma_{\alpha}$ of \eqref{originalsystem} that intersects the $x$-axis at $(-\alpha, 0)$
is periodic. Using \eqref{originalsystem} and the first integral \eqref{V} the time taken 
from the point of intersection  of $\gamma_{\alpha}$ with the negative-$x$ semi-axis, 
$(-\alpha, 0),$ to the first intersection with the positive-$y$ semi-axis, 
occurring at the point $(0,\sqrt{2}\sin\alpha),$ is given by the following time map

\begin{equation}
T(\alpha) := \int_0^\alpha\frac{1}{\sqrt{\cos 2x-\cos 2\alpha}}dx. \label{timemapT}
\end{equation}

We will also need to measure the time taken by $\gamma_{\alpha}$ described above
between its point of intersection with the positive-$y$ semi-axis, $(0,\sqrt{2}\sin\alpha),$
and the point of its first intersection with the vertical line $x=\nu$, with $\nu\in(0,\tfrac{\pi}{2}).$
In the same way as above, the fact that \eqref{V} is a first integral allows us to conclude that
this time is given by the time map

\begin{equation}
T_1(\alpha,\nu) := \int_0^\nu\frac{1}{\sqrt{\cos 2x-\cos 2\alpha}}dx. \label{timemapT1}
\end{equation}

Observe that $T_1(\phi,\phi) = T(\phi).$

\medskip

The proof of the following result can be consulted in \cite{cggp}.

\begin{prop}\label{proptimemaps}
	Let $0<\phi<\alpha<\frac{\pi}{2}.$ The time maps $T$ and $T_1$ defined by
	\eqref{timemapT} and \eqref{timemapT1}, respectively, satisfy:
	\begin{enumerate}
		\item $\alpha\mapsto T(\alpha)$ is strictly increasing, and converges to $+\infty$ as
		$\alpha\to \frac{\pi}{2}$ and to $\frac{\pi}{2\sqrt{2}}$ as $\alpha\to 0.$
		\item $\alpha\mapsto T_1(\alpha,\phi)$ is strictly decreasing.
	\end{enumerate}
\end{prop}

To study the orbits located above the homoclinic orbit to $(\tfrac{\pi}{2},0)\equiv(-\tfrac{\pi}{2},0)$
in the positive-$y$ semi-plane we use as an identifying parameter its intersection with some
positive line (instead of the parameter $\alpha$ above that in these cases is nonexistent, since these
orbits do no intersect the $x$-axis). In \cite{cggp,cmp} the parameter used in these cases was the
ordinate $\beta$ of the intersection of the orbit with the positive-$y$ semi-axis. Here
we shall use as parameter the value $z=y(-L)^2$, i.e., the square of the intersection of the
orbit with the vertical line $x= -\phi$, or, in terms of the original boundary value problem,
the square of the value of the derivative of the solution $x(t)$ at the boundary point $t=-L$. 
For orbits intersecting the $x$-axis we can easily relate the parameters $\alpha$ and $z$ 
using the first integral \eqref{V}:
$V(-\alpha, 0) = V(-\phi, \sqrt{z}).$ In order not to overload the notation we shall use the same
symbols, $T$ or $T_1$, for the time maps independently of which variable, $\alpha$ or $z$,
is being used in the 
parametrization of the orbits.

\section{Phase space analysis of orbits bifurcating from $\gamma^*$}\label{sec:phaseplane*}

Let $\gamma^*$ be the orbit of \eqref{originalsystem}--\eqref{originalbc} shown
in Figure~\ref{fig1}. Let $T^*=2L^*$ be the time taken by this orbit.
This orbit rests on the periodic orbit of \eqref{originalsystem} intersecting the negative
$x$-axis at $x=-\phi.$ Slightly perturbing this supporting periodic orbit to another
whose intersection with the negative $x$-axis is at $-\alpha < -\phi$, with $\alpha-\phi$ 
sufficiently small, we easily conclude from the phase portrait and from the
continuous dependence of solutions of ODEs on the initial data over finite time intervals, 
that there exists four distinct orbits satisfying \eqref{originalsystem}--\eqref{originalbc} 
for appropriately chosen values of $L$ close to $L^*$. We shall denote these
as solutions of type $I$, $A$, $B$, and $C$, as illustrated in Figure~\ref{types}.
%
%
%
\begin{figure}[h]
	\psfragfig*[scale=0.50]{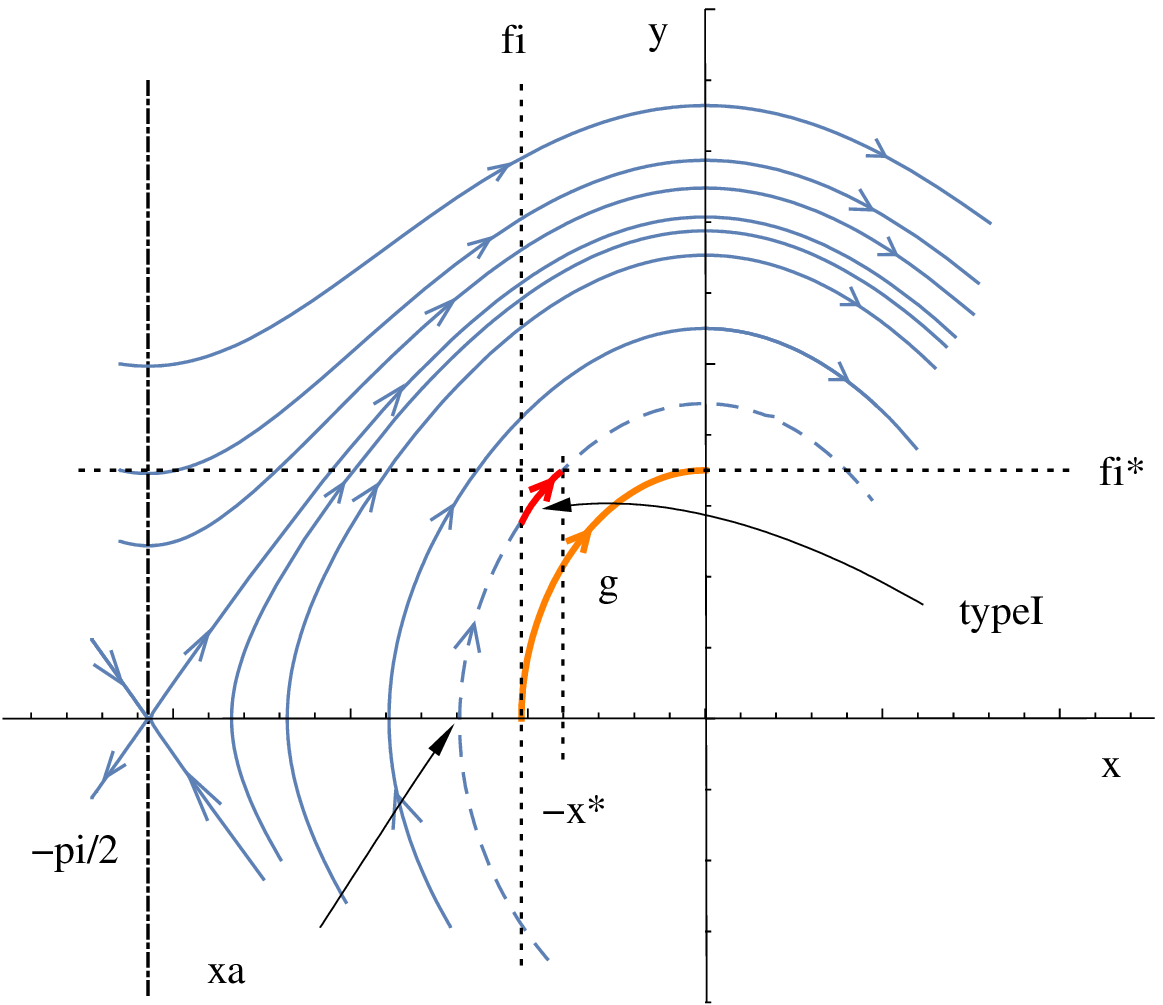}{%
		\psfrag{x}{$x$}
		\psfrag{y}{$y$}
		\psfrag{fi}{$-\phi$}
		\psfrag{fi*}{$\phi^\star$}
		\psfrag{x*}{$x^\star$}
		\psfrag{-x*}{$-x^\star$}
		\psfrag{g}{$\gamma^*$}
		\psfrag{xa}{$x=-\alpha$}
		\psfrag{typeA}{$\text{type A}$}
		\psfrag{typeB}{$\text{type B}$}
		\psfrag{typeC}{$\text{type C}$}
		\psfrag{typeI}{$\text{type I}$}
		\psfrag{-pi/2}{$-\frac{\pi}{2}$}
	}
	\psfragfig*[scale=0.50]{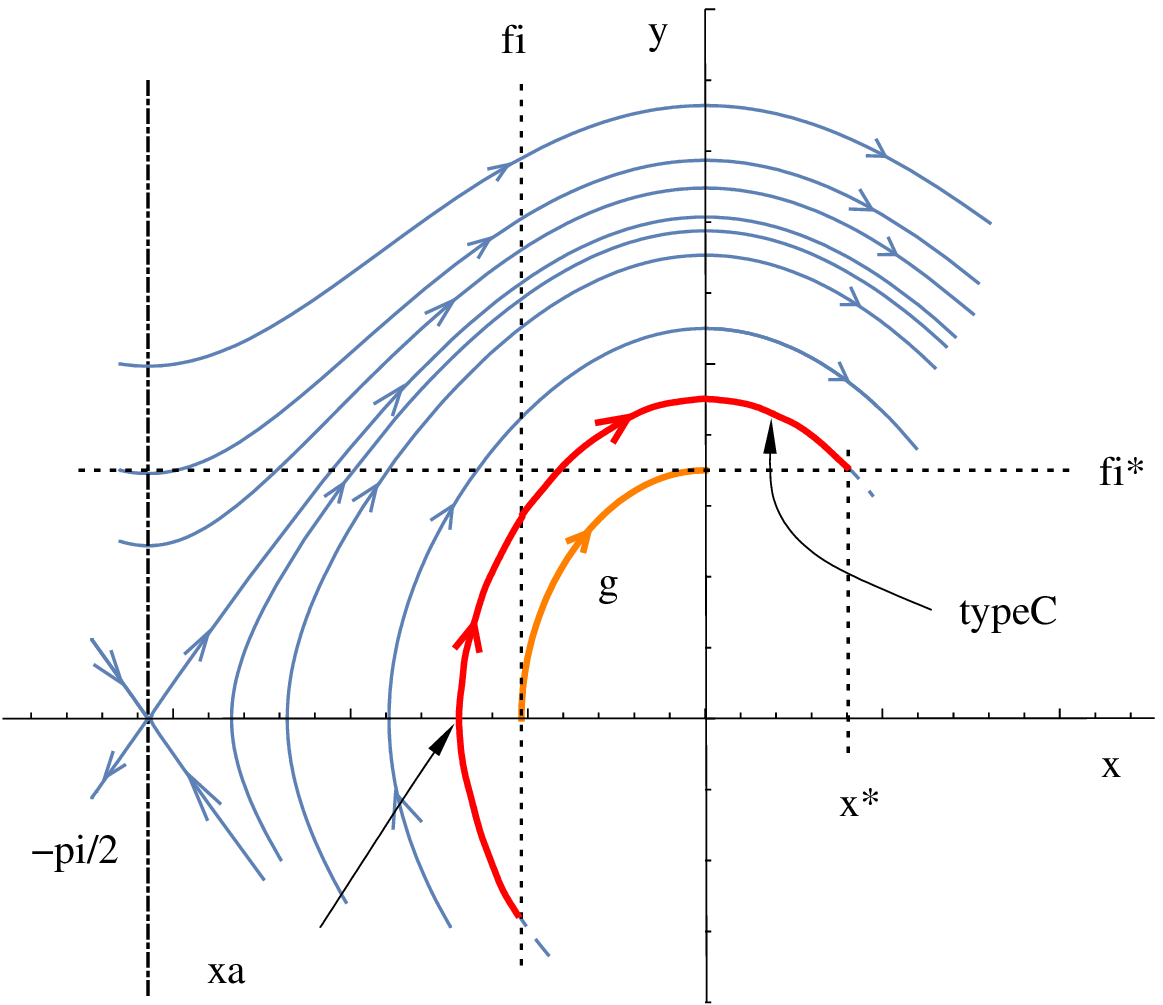}{%
		\psfrag{x}{$x$}
		\psfrag{y}{$y$}
		\psfrag{fi}{$-\phi$}
		\psfrag{fi*}{$\phi^\star$}
		\psfrag{x*}{$x^\star$}
		\psfrag{-x*}{$-x^\star$}
		\psfrag{g}{$\gamma^*$}
		\psfrag{xa}{$x=-\alpha$}
		\psfrag{typeA}{$\text{type A}$}
		\psfrag{typeB}{$\text{type B}$}
		\psfrag{typeC}{$\text{type C}$}
		\psfrag{typeI}{$\text{type I}$}
		\psfrag{-pi/2}{$-\frac{\pi}{2}$}
	}
	\psfragfig*[scale=0.50]{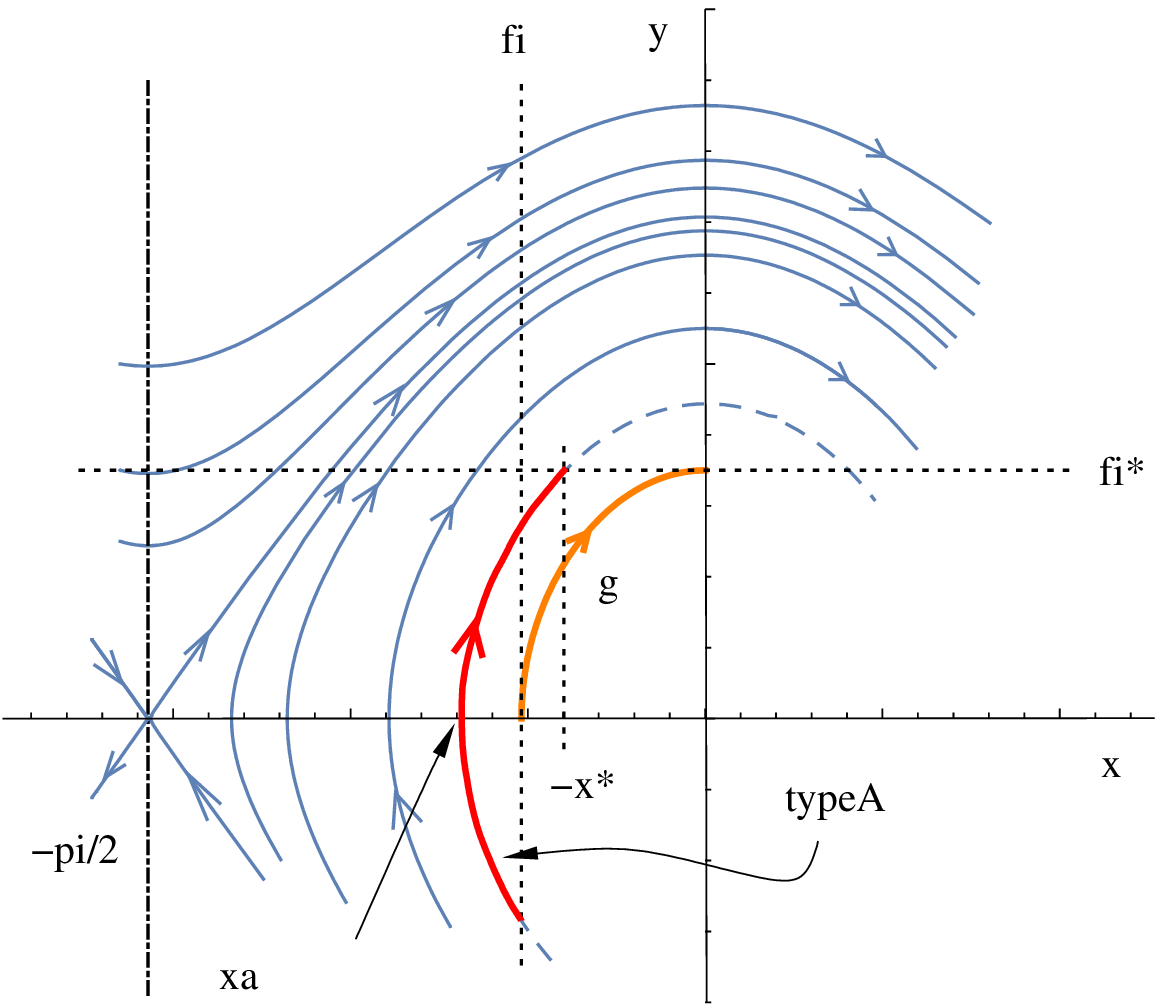}{%
		\psfrag{x}{$x$}
		\psfrag{y}{$y$}
		\psfrag{fi}{$-\phi$}
		\psfrag{fi*}{$\phi^\star$}
		\psfrag{x*}{$x^\star$}
		\psfrag{-x*}{$-x^\star$}
		\psfrag{g}{$\gamma^*$}
		\psfrag{xa}{$x=-\alpha$}
		\psfrag{typeA}{$\text{type A}$}
		\psfrag{typeB}{$\text{type B}$}
		\psfrag{typeC}{$\text{type C}$}
		\psfrag{typeI}{$\text{type I}$}
		\psfrag{-pi/2}{$-\frac{\pi}{2}$}
	}		
	\psfragfig*[scale=0.50]{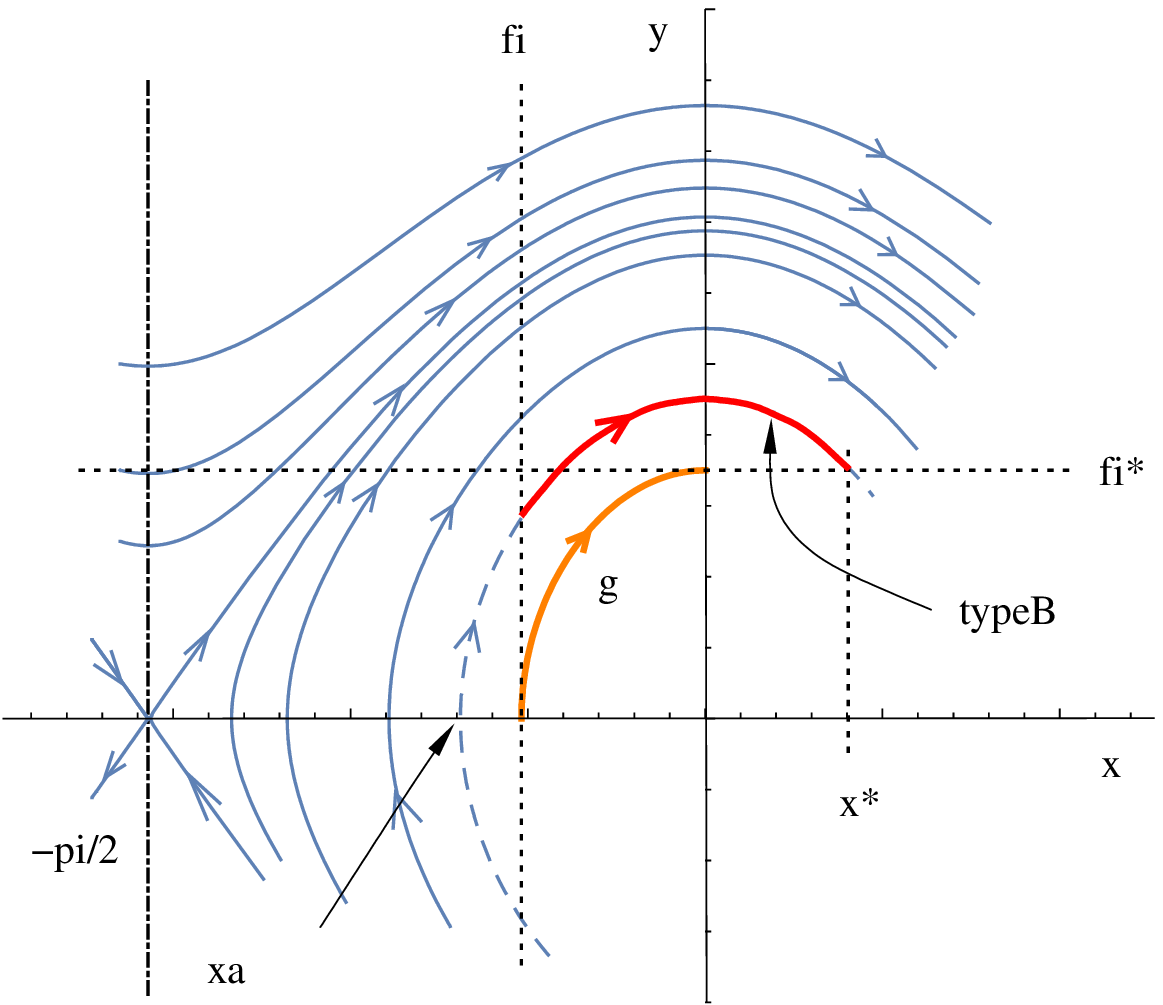}{%
		\psfrag{x}{$x$}
		\psfrag{y}{$y$}
		\psfrag{fi}{$-\phi$}
		\psfrag{fi*}{$\phi^\star$}
		\psfrag{x*}{$x^\star$}
		\psfrag{-x*}{$-x^\star$}
		\psfrag{g}{$\gamma^*$}
		\psfrag{xa}{$x=-\alpha$}
		\psfrag{typeA}{$\text{type A}$}
		\psfrag{typeB}{$\text{type B}$}
		\psfrag{typeC}{$\text{type C}$}
		\psfrag{typeI}{$\text{type I}$}
		\psfrag{-pi/2}{$-\frac{\pi}{2}$}
	}		
	\caption{Phase space illustration of the four types of orbits of \eqref{originalsystem}--\eqref{originalbc}
		obtained by a perturbation of the orbit $\gamma^*$, denoted by $A, B, C,$ and $I$.
		}\label{types}
\end{figure}
%
%
%

Taking into account the time maps defined in Section~\ref{sec:timemaps} the time spent
in each orbit of the above types is given by, respectively,

\begin{eqnarray}
T_I(\alpha,\phi) & := & T_1(\alpha, \phi) - T_1(\alpha, x^\star) \label{TimeI} \\
T_A(\alpha,\phi) & := & 2T(\alpha) - \left(T_1(\alpha, \phi) + T_1(\alpha, x^\star)\right)\label{TimeA} \\
T_B(\alpha,\phi) & := & T_1(\alpha, \phi) + T_1(\alpha, x^\star)\label{TimeB} \\
T_C(\alpha,\phi) & := & 2T(\alpha) - \left( T_1(\alpha, \phi)-T_1(\alpha, x^\star) \right) \label{TimeC} 
\end{eqnarray}
where $x^\star$ is the positive solution of $V(x^\star,\phi^\star)=V(-\alpha,0),$ i.e., 
\begin{equation}
x^\star = x^\star(\alpha,\phi) := \frac{1}{2}\arccos(1 - \cos 2\phi +\cos 2\alpha).\label{xstaralpha}
\end{equation}

From the phase portraits in figures~\ref{fig1} and \ref{types} we conclude that
 orbits of type $A$ and $C$ can be continued down to $-\alpha \downarrow -\frac{\pi}{2}$,
which corresponds to their initial point converging to points on the homoclinic orbit to $(-\frac{\pi}{2},0)\equiv(\frac{\pi}{2},0)$ in $\{y<0\},$ but not further down: if the
initial point gets to, or below, this homoclinic orbit the corresponding orbit remains
in $\{y<0\}$, and the solution will not satisfy the boundary condition $y(L)=\phi^\star>0$, for
any value of $L$.

\medskip

In contradistinction with these cases, in principle there is no obstruction to orbits of types 
$I$ and $B$ to be continued above the homoclinic orbit to $(-\frac{\pi}{2},0)\equiv(\frac{\pi}{2},0)$ in $\{y>0\}.$ To properly handle this possibility it is convenient to parameterize the orbits, and the corresponding time maps, not
by $\alpha$ but by either the ordinate of its intersection with the positive $y$-axis,
$\beta$, or by the ordinate of its initial point $y(-L)$, or, as we shall use 
in Section~\ref{subsec:globalB}, by the square of this quantity $z:=y(-L)^2$. 
Using these parameterizations the variable $\alpha$ in the function $x^\star$ needs
to be correspondingly changed to $\beta$, $y(-L)$, or $z$, which is easily done using 
the fact that $V$ is a first integral to relate the various parameters, 
$V(-\alpha,0) = V(0,\beta)=V(-\phi, y(-L))=V(-\phi, \sqrt{z}),$ leading to the
corresponding expressions
for $x^\star$. One that we shall frequently use in what follows is the expression
in terms of $z$:   
\begin{equation}
\bx(z) := \frac{1}{2}\arccos(1-z).\label{xstarz}
\end{equation}

\medskip

Notwithstanding the possibility of these orbits to be continued above the homoclinic, 
they cannot be continued for arbitrarily large values of $y(-L)$, although
for different reasons, as we shall see below.

\medskip

Type $I$ orbits obviously cease to exist when $y(-L)=\phi^\star$, since, when this occurs,
the initial and final points of the orbit coincide (see Figure~\ref{figtypeIend}.) 
%
%
%
\begin{figure}[h]
	\psfragfig*[scale=0.85]{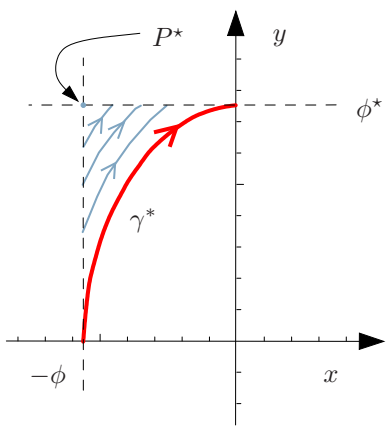}{%
			\psfrag{x}{$x$}
			\psfrag{y}{$y$}
			\psfrag{P}{$P^\star$}			
			\psfrag{x=-}{$-\phi$}
			\psfrag{x=+}{$\phi^\star$}
			\psfrag{g}{$\gamma^*$}
	}
	\caption{Continuation of type $I$ orbits of \eqref{originalsystem}--\eqref{originalbc}
		for the initial point of the orbit with increasing values of $y(-L)$, 
		showing that, when $y(-L)\to \phi^\star$, these 
		orbits converge to a single point $P^\star=(-\phi, \phi^\star)$ and then
		vanish.
	}\label{figtypeIend}
\end{figure}
%
%
%
The
reason why these orbits cannot be continued above this value of $y(-L)$ is easy to understand
from the phase portrait: 
since the initial and final points of type $I$ orbits are always regular points 
of the phase plane, having $y(-L)$ approaching the limit value $\phi^\star$ we have type $I$ 
orbits taking less and less time $2L$, with $L\to 0$ as $y(-L)\to \phi^\star$,
and thus \eqref{originalsystem}--\eqref{originalbc} having no sense in the limit.
In fact, analysis of the time maps tell us exactly the same: 
taking $z\to (\phi^\star)^2$ in \eqref{TimeI} (with the
variable $z$ instead of $\alpha$) and noting that, by \eqref{xstarz}
and the definition of $\phi^\star$, 
$\displaystyle{\lim_{z\to (\phi^\star)^2}\bx(z)	= \phi}$, it immediately follows that $T_I(z)\to 0$. In section~\ref{subsec:typeI} 
a study of the monotonicity of $T_I$ will be presented. 

\medskip

The situation for type $B$ orbits is more interesting. Since all orbits of \eqref{originalsystem}
above the orbit homoclinic to $(-\frac{\pi}{2},0)\equiv(\frac{\pi}{2},0)$ in $\{y>0\}$ have
an absolute minimum at $x=\frac{\pi}{2}$, there are no type $B$ orbits
resting on an orbit of \eqref{originalsystem}
if the $y$-component of that minimum is bigger that $\phi^\star$,
since in this case no segment of the orbit (and in particular the one we call 
type $B$ orbit) can satisfy the boundary condition $y(L)=\phi^\star$. 
Using
the first integral $V$ this means that the largest value of $y(-L)$ that a
type $B$ orbit must satisfy is given by
$V(-\phi, y(-L)) = V(\frac{\pi}{2},\phi^\star)$, and so, from 
$\phi^\star = \sqrt{1-\cos 2\phi},$ we must
have $y(-L)= \sqrt{2},$ independently of $\phi.$ 
This implies that orbits of type $B$ only exist
for $y(-L)\in (0, \sqrt{2}]$. 

\medskip

To understand what is going on in this case we observe that,
 due to the periodicity of the vector field, for initial
points $(-\phi, y(-L))$ with $y(-L)$ bigger than
the ordinate of the point on the homoclinic orbit (but less than $\sqrt{2}$), there
is another orbit with end point $(-x^\star, \phi^\star)$. This orbit is part of a new class
of orbits we shall call type $B'$. See Figure~\ref{figtypeBB'end}.
%
%
%
\begin{figure}[h]
	\psfragfig*[scale=0.55]{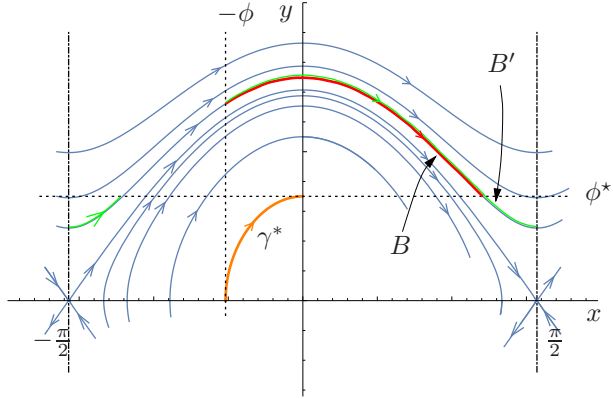}{%
		\psfrag{x}{$x$}
		\psfrag{y}{$y$}
		\psfrag{B}{$B$}
		\psfrag{B1}{$B'$}
		\psfrag{fi}{$-\phi$}
		\psfrag{fi*}{$\phi^\star$}
		\psfrag{g}{$\gamma^*$}
		\psfrag{pi/2}{$\frac{\pi}{2}$}
		\psfrag{-pi/2}{$-\frac{\pi}{2}$}
	}
	\caption{An orbit of \eqref{originalsystem}--\eqref{originalbc} of
		type $B$ above the homoclinic orbit and the orbit of type $B'$
		with the same $y(-L)$.
	}\label{figtypeBB'end}
\end{figure}
%
%
%
When $y(-L)\to \sqrt{2}$ the two end points of orbits of types $B$ and $B'$ 
converge to one another and at $y(-L)=\sqrt{2}$ the two coincide
with the point $(\frac{\pi}{2},\phi^\star)$, and both cease to exist 
for $y(-L)>\sqrt{2}.$

\section{Bifurcation diagram of orbits bifurcating from $\gamma^\star$}\label{sec:bifurcation*}

To draw the bifurcation diagram of orbits bifurcating from $\gamma^\star$ we
need to put together the information in Section~\ref{sec:phaseplane*}, 
gathered from the phase portrait, with information about the time spent by each orbit,
obtained from the study of the time maps, which we will do next.

\subsection{Behavior of type $I$ solutions branch}\label{subsec:typeI}

We first consider solutions of type $I$. 
From \ref{TimeI}, the definition of the time maps \eqref{timemapT} and \eqref{timemapT1}, 
and Proposition~\ref{proptimemaps}, we conclude that 
\[
T_I(\alpha,\phi) = T_1(\alpha, \phi) - T_1(\alpha, x^\star) < T_1(\alpha, \phi) < T_1(\phi, \phi)  =
T(\phi) = T^* = 2L^*.
\] 
Thus, in the bifurcation diagram plotted using the time spent by the orbit as the
bifurcation parameter, type $I$ branch of solutions exist to the left of the bifurcation 
point $T^*$ correspondent to the critical orbit $\gamma^\star.$
Let us compute the derivative $\partial T_I/\partial \alpha.$ From 
\begin{eqnarray*}
	\frac{\partial T_I}{\partial \alpha}(\alpha,\phi)  
	& = &  
	\frac{\partial T_1}{\partial \alpha}(\alpha,\phi) - 
	\frac{\partial T_1}{\partial \alpha}(\alpha,x^\star) -
	\frac{\partial T_1}{\partial x^\star}\frac{dx^\star}{d\alpha} \\
	& = & -  \int_{x^\star(\alpha)}^{\phi}\frac{\sin 2\alpha}{(\cos 2x-\cos 2\phi)^{\frac{3}{2}}}dx \\
	\!\!\!\!&  & \!\! - \;
	\frac{2\sin 2\alpha}{\sqrt{\cos 2x^\star(\alpha)-\cos 2\phi}\sqrt{1-(1-\cos 2\phi - \cos 2\alpha)^2}},
\end{eqnarray*}	
we conclude that $\partial T_I/\partial \alpha <0$, since it is clear from the definition 
of type $I$ solutions that we always have $x^\star < \phi$
(see Figure~\ref{figtypeIend}). This means that the branch of type $I$ solutions 
in the bifurcation diagram has no turning points.

\medskip

The above computations were done using the parametrization of orbits by
the parameter $\alpha$, and thus the corresponding orbits are inside the region 
bounded by the homoclinics. This is always the case when $(-\phi, \phi^\star)$
is in this region. When it is outside this region the orbits can still be continued,
as explained in Section~\ref{sec:phaseplane*}, and the results above still hold
using a parametrization of the orbits by either of the parameters
introduced therein, namely $y(-L), z$, or $\beta$.

\medskip

The results above and the discussion in Section~\ref{sec:phaseplane*} allows us to conclude
that the type $I$ solutions branch continues monotonically to $T=0$, as shown in 
Figure~\ref{localbif}.

\subsection{Behavior of type $C$ solutions branch}\label{subsec:typeC}

Consider now solutions of type $C$. 
From \eqref{TimeC}, the definition of the time maps \eqref{timemapT} and \eqref{timemapT1}, 
and Proposition~\ref{proptimemaps}, we conclude that 
\begin{eqnarray*}
T_C(\alpha,\phi) & = & 2T(\alpha) + T_1(\alpha, x^\star) - T_1(\alpha, \phi) \\
& = & T(\alpha) + \underbrace{T_1(\alpha, x^\star)}_{>0} + 
\underbrace{\bigl(T(\alpha) - T_1(\alpha, \phi)\bigr)}_{>0} \\
& > & T(\alpha) \\
& > & T(\phi) = T^* = 2L^*.
\end{eqnarray*}

In the other hand, since by \eqref{TimeI} and \eqref{TimeC} we can write $T_C(\alpha) = 2T(\alpha) -
T_I(\alpha)$, we conclude that
\[
\frac{\partial T_C}{\partial \alpha}(\alpha,\phi) = 2T'(\alpha) - \frac{\partial T_I}{\partial \alpha}(\alpha,\phi)
> 0,
\]
where the positivity comes from Proposition~\ref{proptimemaps} and the result in
Section~\ref{subsec:typeI}. Note that orbits of type $C$ are always inside the 
region bounded by the homoclinics and so this analysis is enough to conclude that, like the 
branch of type $I$ solutions, the type $C$ solutions branch do not have turning points
and, from Proposition~\ref{proptimemaps}.(1), exists globally when $L\to +\infty$, since
type $C$ orbits take progressively longer times as $\alpha\to\pi/2.$.

\subsection{Local behavior of type $B$ solutions branch}\label{subsec:localB}

In this section we study the behavior of the solution branch
of type $B$ solutions locally close to the bifurcation point. 

\medskip

Consider the branch of bifurcating solutions of \eqref{originalsystem}--\eqref{originalbc}
denoted by $B$ in Section~\ref{sec:phaseplane*}. As already observed, the 
time spent by an orbit of type $B$ is given by
\begin{equation}
T_B(\alpha,\phi) = T_1(\alpha,\phi) + T_1(\alpha, x^\star(\alpha)),
\end{equation}
where $T_1$ is the time map defined by \eqref{timemapT1}
and
$x^\star$ is defined by \eqref{xstaralpha}.
Please see the plot of a type $B$ orbit in Figure~\ref{types} in order to clarify this notation.

\medskip

Since type $B$ solutions can be continued above the homoclinic orbit to 
$(-\frac{\pi}{2},0)\equiv(\frac{\pi}{2},0)$ in $\{y>0\}$ it is natural to consider the
orbit parametrized by the ordinate of one of its points. It turns out that, from the computational
point of view, an appropriate parameter is the square of the ordinate $y(-L)$ of the 
initial point of the orbit. We shall denote this parameter by $z$. Using 
$V(-\phi,\sqrt{z})=V(-\alpha,0)$, we can obtain the expression for $T_B$ from \eqref{TimeB}
when the orbit is bounded by the homoclinics and extended to larger values of $z$ as explained 
in Section~\ref{sec:timemaps}. We thus have
\begin{equation}
T_B(z,\phi)= \int_0^{\phi}(z-\cos 2\phi+\cos 2x)^{-1/2}dx + 
\int_0^{\bx(z)}(z-\cos 2\phi+\cos 2x)^{-1/2}dx\label{TBzz}
\end{equation}
where  $z \in [0, 2],$   $\phi \in (0, \frac{\pi}{2})$, and $\bx(z)$ 
is the function $x^\star$ expressed in the new variable $z$, 
defined above in \eqref{xstarz}.

\subsubsection{Computations of $\frac{\partial T_B}{\partial z}$}\label{subsubsec:dTB/dbeta}

We want to prove that for $z$ sufficiently close to zero the time map satisfies $T_B(z,\phi) < T_B(0,\phi)$. 
To achieve this, we prove, in Proposition~\ref{prop3}, that $\frac{\partial T_B}{\partial z}(0) = -\infty$, which obviously implies the inequality.

\begin{prop}\label{prop1}
	$\frac{\partial T_B}{\partial z}(z, \phi) \longrightarrow -\infty$ as $z \to 0.$
\end{prop}

\begin{proof}
	Differentiating \eqref{TBzz} with respect to $z$ we get
	\begin{eqnarray}
	2\frac{\partial T_B}{\partial z}(z,\phi) & = & - \int_0^{\phi}(z-\cos 2\phi+\cos 2x)^{-3/2}dx  \,- \nonumber \\
	& & -\, \int_0^{\bx(z)}(z-\cos 2\phi+\cos 2x)^{-3/2}dx \,+ \nonumber \\
	& & +\, 2(z - \cos 2\phi + \cos 2\bx(z))^{-\frac{1}{2}}\frac{d\bx}{dz}(z)\nonumber \\
	& = &
	- \int_0^{\phi}(z-\cos 2\phi+\cos 2x)^{-3/2}dx  \,- \nonumber \\
	& & -\, \int_0^{\bx(z)}(z-\cos 2\phi+\cos 2x)^{-3/2}dx \,+ \label{dTBdz} \\
	& & +\,
	z^{-\frac{1}{2}}(2-z)^{-\frac{1}{2}}(1 - \cos 2\phi)^{-\frac{1}{2}}.\nonumber
	\end{eqnarray}

\noindent
We now estimate the integral terms in this expression, starting with the second integral.
We first need to look at the behavior of $z\mapsto \bx(z)$:  a simple application of the 
following generalized Taylor expansion 
\begin{equation}
\arccos(1-x) = \sqrt{2x} + \frac{(2x)^{\frac{3}{2}}}{24} + O(x^2)\quad\text{as}\quad x\to 0\label{gte}
\end{equation}
allows us to write 
\begin{equation}
\bx(z) = \frac{1}{\sqrt{2}}z^{\frac{1}{2}} + \frac{\sqrt{2}}{24}z^{\frac{3}{2}} + O(z^2), 
\quad\text{as $z\to 0$.}\label{expansionxstar}
\end{equation}

\begin{lemma}\label{I2} 
	${\displaystyle{\int_0^{\bx(z)}(z-\cos 2\phi+\cos 2x)^{-3/2}dx }}
	= O\bigl(z^{\frac{1}{2}}\bigr)\quad\text{as}\quad z\to 0.$
\end{lemma}

\begin{proof}
Since $1-\cos 2\phi = z-\cos 2\phi + \cos 2 x^\star(z)< z-\cos 2\phi + \cos 2x < z+2,$ and using \eqref{expansionxstar}, we get, as $z\to 0$,
\[
\frac{1}{4}z^{\frac{1}{2}} +  O\bigl(z^{\frac{3}{2}}\bigr) < \int_0^{\bx(z)}(z-\cos 2\phi+\cos 2x)^{-3/2}dx \leqs \frac{1}{\sqrt{2}}(1-\cos 2\phi)^{-\frac{3}{2}}z^{\frac{1}{2}} +  O\bigl(z^{\frac{3}{2}}\bigr),
\]
which proves the lemma.
\end{proof}

\begin{lemma}\label{L6}
	${\displaystyle{\int_0^{\phi}(z-\cos 2\phi+\cos 2x)^{-3/2}dx }}
		= z^{-\frac{1}{2}}\left(\frac{1}{\sin 2 \phi} + o(1)\right) \quad\text{as}\quad z\to 0.$
\end{lemma}

\begin{proof}
	Consider the trigonometric identity 
	\begin{equation}
	\cos a -\cos b = 2\sin\frac{b+a}{2}\sin\frac{b-a}{2}\label{trigo}
	\end{equation}
	and write the integral as 
	\begin{equation}
	\int_0^{\phi}(z+2\sin(\phi+x)\sin(\phi-x))^{-3/2}dx.
	\nonumber 
	\end{equation}	
Considering the change of variable $x\mapsto t$ defined by $zt = \sin(\phi-x)$ we have
$\frac{dt}{dx} = - \frac{1}{z}\sqrt{1-(zt)^2},$ and the integral becomes
\begin{eqnarray}
\lefteqn{z^{-\frac{1}{2}}\int_0^{\frac{\sin\phi}{z}}\frac{1}{\sqrt{1-(zt)^2}\,\Bigl(1+2t\sin\bigl(2\phi - \arcsin(zt)\bigr)\Bigr)^{\frac{3}{2}}}\,dt \;\, =}  \nonumber \\
& = & z^{-\frac{1}{2}}\int_0^{+\infty}\underbrace{\frac{\text{\Large $\mathbf{1}$}_{\bigl(0,\frac{\sin\phi}{z}\bigr)}(t)}{\sqrt{1-(zt)^2}\,\Bigl(1+2t\sin\bigl(2\phi - \arcsin(zt)\bigr)\Bigr)^{\frac{3}{2}}}}_{=: f(z,t)}\,dt \label{integralz}
\end{eqnarray}
Observe now that
\begin{equation}
\lim_{z\to 0}f(z,t) = \frac{1}{\Bigl(1+2t\sin 2\phi\Bigr)^{\frac{3}{2}}} =: f(t),\quad\text{pointwise in $t$,}\nonumber
\end{equation}
and note also that, since $0\leqs zt \leqs \sin\phi$ and $0 < \phi < \frac{\pi}{2}$,
we conclude that $\sin\bigl(2\phi - \arcsin(zt)\bigr)$ is positive and bounded
away from zero satisfying\footnote{Recall the notation $a\wedge b = \min\{a, b\},$} 
$\sin\bigl(2\phi - \arcsin(zt)\bigr) > \sin 2\phi \wedge \sin \phi
=: m(\phi) > 0.$
Also $\sqrt{1- (zt)^2} \geqs \sqrt{1-\sin^2\phi} = \cos\phi >0.$ From these it follows that
\begin{eqnarray}
0\; < \;f(z,t) & \leqs & \frac{1}{\cos\phi}\frac{\text{\Large $\mathbf{1}$}_{\bigl(0,\frac{\sin\phi}{z}\bigr)}(t)}
{\Bigl(1+2tm(\phi)\Bigr)^{\frac{3}{2}}} \nonumber \\
& \leqs & \frac{1}{\cos\phi}\frac{1}{\Bigl(1+2tm(\phi)\Bigr)^{\frac{3}{2}}}\; =: \; g(t)\nonumber 
\end{eqnarray}
and $g$ is integrable in $[0, +\infty).$ Hence, by the 
Lebesgue's dominated convergence theorem applied to the integral in
\eqref{integralz} we conclude that, as $z\to 0,$
\[
\int_0^{+\infty}f(z,t)dt  \;=\; \int_0^{+\infty}\frac{1}{\Bigl(1+2t\sin 2\phi\Bigr)^{\frac{3}{2}}}dt + o(1) \;=\; \frac{1}{\sin 2\phi} + o(1),
\]
and this concludes the proof of the lemma.
\end{proof}

Collecting the results in the above lemmas we can write, as $z \to 0,$
\begin{eqnarray}
2\frac{\partial T_B}{\partial z}(z,\phi) 
& = &
-z^{-\frac{1}{2}}\left(\frac{1}{\sin 2\phi} + o(1)\right) 
+ z^{-\frac{1}{2}}\frac{1}{\sqrt{2}}(1-\cos 2\phi)^{-\frac{1}{2}} + O\bigl(z^{\frac{1}{2}}\bigr) \nonumber \\
& = & - z^{-\frac{1}{2}}\left(\frac{1}{\sin 2\phi} - \frac{1}{\sqrt{2}(1-\cos 2\phi)^{\frac{1}{2}}} + o(1)\right) \nonumber \\
& = & 
- z^{-\frac{1}{2}}\left(\frac{\sqrt{2}(1-\cos 2\phi)^{\frac{1}{2}}-\sin 2\phi}{\sqrt{2}\sin 2\phi(1-\cos 2\phi)^{\frac{1}{2}}} + o(1)\right). \label{finalestimate}
\end{eqnarray}
Consider the function defined in the interval $[0,\pi)$ by $h(x) = 2(1-\cos x) - \sin^2x$.
Clearly $h(0)=0$ and $h'(x) >0$ for $x\in(0,\pi)$. Thus $h(x) >0$ if $x\in (0, \pi)$.
This implies that $2(1-\cos 2\phi)>\sin^22\phi$ for all $\phi\in (0, \frac{\pi}{2})$
and the monotonicity of $\sqrt{\cdot}$ entails $\sqrt{2}(1-\cos 2\phi)^{\frac{1}{2}}>\sin2\phi.$
Using this in \eqref{finalestimate} completes the proof of Proposition~\ref{prop1}.
\end{proof}

\medskip

The next proposition is an elementary result in Real Analysis that we state and prove for completeness

\begin{prop}\label{prop2}
	Let $\psi: [0, c)\to \Rb$ be a function continuous in $[0,c)$, differentiable in $(0,c)$,
	and assume that $\psi'(x)\to -\infty$ as $x\to 0$. Then, the following holds true:
	 $\psi'(0):= {\displaystyle \lim_{\xi\to 0^+}\frac{\psi(\xi)-\psi(0)}{\xi} = -\infty}.$
\end{prop}	

\begin{proof}
	Fix $\xi\in(0,c)$ and apply the mean value theorem to the interval $[0,\xi]$. We conclude that
	there exists $x \in (0, \xi)$ such that
	\[
	\frac{\psi(\xi)-\psi(0)}{\xi} = \psi'(x).
	\]
	Passing to the limit as $\xi\to 0$ in both sides of this expression, 
	and using the assumption
	about $\psi'(x)$ when $x$ approaches $0$ in the right-hand side, 
	we conclude the proof.
\end{proof}

\medskip

We can now apply Propositions~\ref{prop1} and \ref{prop2} to immediately conclude that

\begin{prop}\label{prop3}
	 $\frac{\partial T_B}{\partial z}(0,\phi) = -\infty.$
\end{prop}

The results of Propositions~\ref{prop1} and \ref{prop3} imply that the time taken by a
solution of type $B$ close to the bifurcation point is smaller than the time $T^*$ taken
by the critical solution $\gamma^\star$.

\medskip

Actually, for the study of type $B$ solutions it would have been enough to prove that
the derivative $\frac{\partial T_B}{\partial z}(z,\phi)$ is negative for
all $z\geqs 0$ sufficiently small. The fact that it is not just a negative real
number is needed for the study of solution branches that correspond to
solutions circling the origin $k$ times, which will be presented in 
section~\ref{sec:otherbifurcations}. We will see that there are solutions analogous to 
those of type $B$ circling the origin $k$ times and taking a time given by 
\[
T_{B_k}(z,\phi)  =  4kT(z) + T_{B}(z,\phi). 
\]
Since $T'(0)\in\Rb^+$, the fact that $\frac{\partial T_B}{\partial z}(0,\phi)<0$
is smaller than any negative real number is what justifies that the bifurcation
diagrams for the ``$k$ branches'' close to their bifurcation points are
qualitatively similar to the case we are presently studying (cf. discussion in section~\ref{sec:otherbifurcations}; see also Figure~\ref{fbifurcations}).

\subsection{On the global behavior of type $B$ solutions branch}\label{subsec:globalB}

The result obtained in the previous section for the time taken by a type $B$ solution is of
a local character: it is valid when the type $B$ orbit is close to the critical one $\gamma^*,$ 
i.e., when the value of the parameter indexing the orbit (be it $\alpha$, $\beta$, or $y(-L)$) is
sufficiently close to the value of the corresponding one in the critical orbit
($\phi$, $\phi^\star$, or $0$, resp.).

\medskip

From the study presented in Section~\ref{sec:phaseplane*} we concluded that the type $B$ branch
of solutions can be continued away from the neighborhood of the critical orbit, and solutions
in this branch, parametrized by the value of $y(-L)$ only cease to exist when the parameter
value is $y(-L)=\sqrt{2}.$ To understand the global behavior of this branch for 
$y(-L)\in (0, \sqrt{2})$  we need to know the behavior of $y(-L)\mapsto T_B(y(-L))$. 
In particular, if we prove that this function is convex, we conclude that the branch of
type $B$ solutions has a unique saddle-node point in the bifurcation diagram.
 
 \medskip
 
As in Section~\ref{subsec:localB}, let us parameterize type $B$ orbits by $z=y(-L)^2$. 

\medskip
 
Differentiating \eqref{dTBdz} with respect to $z$ we get, after some algebraic manipulations,
\begin{eqnarray}
4\frac{\partial^2 T_B}{\partial z^2}(z,\phi) & = &
3 \int_0^{\phi}(z-\cos 2\phi+\cos 2x)^{-5/2}dx  \,+ \nonumber \\
& & +\, 3 \int_0^{\bx(z)}(z-\cos 2\phi+\cos 2x)^{-5/2}dx \,+ \label{d2TBdz2} \\
& & +\,
z^{-\frac{3}{2}}(2-z)^{-\frac{3}{2}}(1 - \cos 2\phi)^{-\frac{3}{2}}g(z,\phi),\nonumber 
\end{eqnarray}
where
\begin{equation}
g(z,\phi) := z^2 - (2\cos 2\phi)z -2(1-\cos 2\phi).\label{g}
\end{equation}

When $g$ is negative, the sign of $\frac{\partial^2 T_B}{\partial z^2}(z,\phi)$ 
depends on the balance between the two positive integrals and the 
(negative) last term in \eqref{d2TBdz2}, and its determination seems to
be a challenging problem.
However, close to the border $z=0$ we can compute the sign of $\frac{\partial^2 T_B}{\partial z^2}$ 
using the asymptotic technique employed in the proof of Lemma~\ref{L6}:
\begin{lemma}\label{L8} $\frac{\partial^2 T_B}{\partial z^2}(z,\phi) > 0$ as $z\to 0.$
\end{lemma}

\begin{proof}
	From \eqref{d2TBdz2} we have
	\begin{equation}
	4\frac{\partial^2 T_B}{\partial z^2}(z,\phi) \geqs
	3 \int_0^{\phi}(z-\cos 2\phi+\cos 2x)^{-5/2}dx  +
	z^{-\frac{3}{2}}(2-z)^{-\frac{3}{2}}(1 - \cos 2\phi)^{-\frac{3}{2}}g(z,\phi).\label{bound1}
	\end{equation}
	Using the trigonometric identity \eqref{trigo} and the
	change of variable $x\mapsto t,$ with  
	$zt=\sin(\phi-x),$ in the integral in \eqref{bound1}, we can write
	\[
	\int_0^{\phi}(z-\cos 2\phi+\cos 2x)^{-\frac{5}{2}}dx = 
	z^{-\frac{3}{2}}\!\!\int_0^{+\infty}\!\!\!\!\!\!\!\!
	\frac{\text{\Large $\mathbf{1}$}_{\bigl(0,\frac{\sin\phi}{z}\bigr)}(t)}{\sqrt{1- (zt)^2}\,\Bigl(1+2t\sin\bigl(2\phi - \arcsin(zt)\bigr)\Bigr)^{\frac{5}{2}}}dt.
	\]
	Observing that the integral in right-hand side is like \eqref{integralz} with $-1/2$ changed to $-3/2$ and $3/2$ to $5/2$, we can apply the argument in the proof of Lemma~\ref{L6} to obtain, as $z\to 0,$
	\begin{equation}
	\int_0^{\phi}(z-\cos 2\phi+\cos 2x)^{-\frac{5}{2}}dx = z^{-\frac{3}{2}}\!\int_0^{+\infty}\!\!\!\frac{1}{(1+2t\sin 2\phi)^{\frac{5}{2}}}\,dt \,+ o(1) = \frac{1}{3\sin 2\phi} + o(1).\label{limitintegral1}
	\end{equation} 
	Now, using the definition of $g$, \eqref{g}, to write  $g(z,\phi) = -2(1-\cos 2\phi) + O(z)$ as 
	$z\to 0$, observing that $\sin 2\phi = (1-\cos^2 2\phi)^{\frac{1}{2}} =
	(1-\cos 2\phi)^{\frac{1}{2}}(1+\cos 2\phi)^{\frac{1}{2}},$ and 
	substituting \eqref{limitintegral1} into \eqref{bound1}, we obtain the following, as $z\to 0$,
	\begin{equation}
	4\frac{\partial^2 T_B}{\partial z^2}(z,\phi) \geqs 
	z^{-\frac{3}{2}}\frac{1}{(1-\cos 2\phi)^{\frac{1}{2}}}\left(\frac{1}{(1+\cos 2\phi)^{\frac{1}{2}}}
			- \frac{1}{\sqrt{2}} + o(1)\right) > 0,\label{limitintegral2}
	\end{equation}
	where the positivity is due to	$1+\cos 2\phi \in (0, 2),$ and hence $(1+\cos 2\phi)^{-\frac{1}{2}}>\frac{1}{\sqrt{2}}.$
\end{proof}  

\medskip

Since the direct handling of \eqref{d2TBdz2} does not seem promising, we tried to
approach the problem of the convexity of $T_B(z,\phi)$
by the method presented in Smoller \cite[Chap. 13\S D]{smoller}: to 
establish the convexity
of $T_B(z,\phi)$ it is sufficient to prove that,
if $(z,\phi)\in (0,2)\times(0,\frac{\pi}{2})$ is a stationary point 
of $z\mapsto T_B(z,\phi)$, we have 
\begin{equation}
\Phi(z,\phi):= 4\frac{\partial^2 T_B}{\partial z^2} + 
2k(z,\phi)\frac{\partial T_B}{\partial z} >0, \label{Phi>0}
\end{equation}
for all functions $k(z,\phi)$, because at these points the value of the first derivative $\frac{\partial T_B}{\partial z}$ 
is zero, by definition. Thus, if we find a function $k(z,\phi)$ such that \eqref{Phi>0} holds for all points
$(z,\phi)$, we conclude that, for each fixed $\phi$, $\Phi(z,\phi)$ is convex at each of its stationary points,
and thus there can exist only one stationary point.
 
 \medskip

Now, by \eqref{dTBdz}, \eqref{d2TBdz2}, and \eqref{g}, choosing 
\begin{align}
k(z,\phi) &:= - z^{-1}(2-z)^{-1}(1-\cos 2\phi)^{-1}g(z,\phi) \\
& = (1-\cos 2\phi)^{-1} - 2\,\frac{z-1}{z(2-z)}, \label{k}
\end{align}
we get
\begin{equation}
\Phi(z,\phi) = \Bigl(\int_0^\phi + \int_0^{\bx(z)}\Bigr)h^{-5/2}(3-kh)\,dx, \label{2int}
\end{equation}
where
\begin{equation}
h = h(z,\phi,x):= z-\cos 2\phi +\cos 2x. \label{h}
\end{equation}

From the definition of $k$ it follows that $k <0$ whenever $g>0$ and, 
from \eqref{2int}, \eqref{h}, and $h(z,\phi, x) < h(z,\phi,0),$ we easily conclude that
 $\Phi(z,\phi) > 0$ in
	$\Omega:= \left\{(z,\phi)\in \Rb^2 \,|\, 3-k(z,\phi)h(z,\phi,0) > 0\right\}.$ 
This set $\Omega$ is illustrated in Figure~\ref{figomega}.

%
%
%
\begin{figure}[h]
	\psfragfig*[scale=0.45]{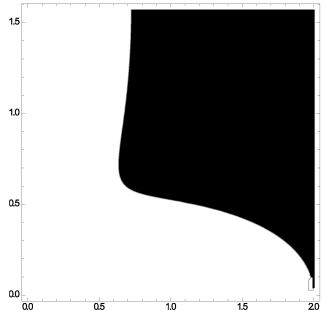}{%
	\psfrag{z}{$z$}
	\psfrag{phi}{$\phi$}
}
	\caption{Illustration  (in black) of the set $\Omega$ of points $(z,\phi)$ where 
		$3-k(z,\phi)h(z,\phi,0) > 0$.}\label{figomega}
\end{figure}
%
%
%

Outside $\Omega$ the sign of $\Phi$ is much harder to establish since the two
integrals can have opposite signs and be divergent in the boundaries of the domain of $\Phi$.
Numerical computations using the software 
Mathematica$\mbox{}^\copyright$ provide very convincing evidence for the positivity of $\Phi(z,\phi)$ everywhere in the rectangle $(0,2)\times (0, \frac{\pi}{2})$, as illustrated in Figure~\ref{figPhi}.

%
%
%
\begin{figure}[h]
	\psfragfig*[scale=0.40]{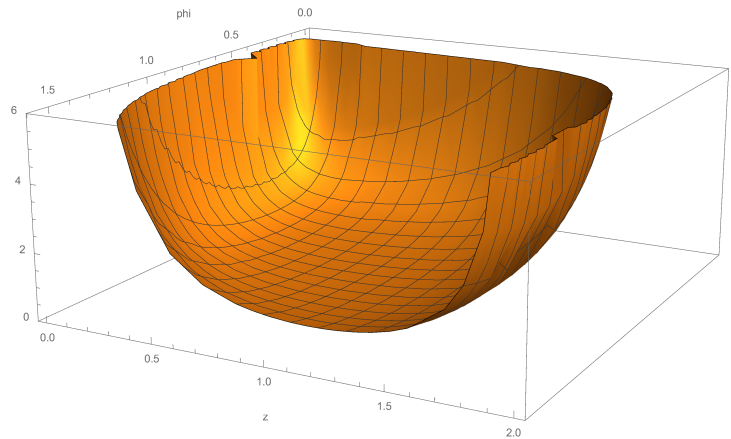}{%
	\psfrag{z}{$z$}
	\psfrag{phi}{$\phi$}
}
	\caption{Graph of  $(z,\phi)\mapsto \Phi(z,\phi)$}\label{figPhi}
\end{figure}
%
%
%

Unfortunately, despite repeated efforts, we were unable to rigorously establish the positivity of $\Phi$ illustrated in Figure~\ref{figPhi}, and hence the convexity of $z\mapsto T_B(z,\phi)$.
Alternative approaches to prove the existence of a single minimum of the graph of
$z\mapsto T_B(z,\phi)$,  based on attempting to define different type of time maps 
via changes of variables \cite{korman,schaaf} or other analytic approaches \cite{korman}
where fruitless.

\medskip

Thus, we state the following

\begin{conj}\label{C1}
	For each $\phi\in (0, \frac{\pi}{2}),$ the function $z\mapsto T_B(z,\phi)$ is convex.
\end{conj}

For the remainder of this paper we assume this conjecture to hold.

\subsection{Behavior of type $A$ solutions branch}\label{subsec:typeA}

From \eqref{TimeA} we know that solutions of type $A$ satisfy 
\(
T_A(\alpha) = 2T(\alpha) - T_B(\alpha).
\) 
By the results of section~\ref{subsec:localB}, we know that $T_B$ is decreasing 
when $\alpha>\phi$ close to $\phi$, and, by Proposition~\ref{proptimemaps}, $T$ is always increasing.
Thus, we conclude that the time $T_A$ taken by
solutions of type $A$ close to the bifurcation 
point is larger than the time $T^*$ taken by the critical solution $\gamma^\star$.

\medskip

By putting together all previous (analytical and numerical) results we 
conclude that the bifurcation diagram of \eqref{originalsystem}--\eqref{originalbc}
about $T^*$ is the one shown in Figure~\ref{localbif}.
%
%
%
\begin{figure}[h]
	\psfragfig*[scale=0.27]{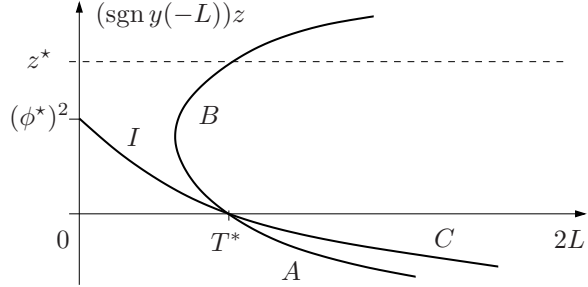}{%
	\psfrag{0}{$0$}
	\psfrag{L}{$2L$}
	\psfrag{y}{$(\sgn y(-L))z$}
	\psfrag{star}{$(\phi^\star)^2$}
	\psfrag{hp}{$z^\star$}
	\psfrag{A}{$A$}
	\psfrag{B}{$B$}
	\psfrag{C}{$C$}
	\psfrag{I}{$I$}
	\psfrag{T0}{$T^*$}
}
	\caption{Local bifurcation around $T^*:=T(\phi)$, and global behaviour of
		the bifurcating branches with $T<T^*$ assuming the behaviour of 
		$B$ orbits stated in Conjecture~\ref{C1}.  
		The plotted case corresponds to $\phi\in \bigl(0,\tfrac{\pi}{4}\bigr).$ 
		When $\phi\in \bigl(\tfrac{\pi}{4}, \tfrac{\pi}{2}\bigr)$ the
		end point of the branch $I$ occurs above the homoclinic orbit, and thus
		$(\phi^\star)^2 > z^\star := 1+\cos 2\phi$.} \label{localbif}
\end{figure}
%
%
%

\section{Other bifurcations}\label{sec:otherbifurcations}

As was the case of system \eqref{originalsystem} with non-homogeneous Dirichlet
boundary conditions studied in \cite{cggp,cmp}, in the present system  
\eqref{originalsystem}--\eqref{originalbc} with large values of $L$ we can have 
solutions turning several times around the origin in the region of the
cylindrical phase space bounded by the two homoclinic orbits.

\medskip
	
To study these cases we use the same principle of perturbing \emph{critical}
orbits $\gamma^*_k$ defined as $\gamma*$ (i.e., satisfying, in addition to \eqref{originalbc}, 
homogeneous boundary conditions $y(-L)=0$ and $x(L)=0$) but now turning $k$ times around the 
origin. We have also four distinct types of solutions that we can denote by $I_k$, $A_k$,
$B_k$ and $C_k$, analogous to $I, A, B$ and $C$, which can be considered the cases with $k=0$
(i.e., orbits that do not have any complete turn around the origin). The
time spent by these orbits with $z$ sufficiently close to zero is obtained by adding $4kT(\alpha)$ to the times spent
by the corresponding $k=0$ orbits, e.g. 
\begin{align}
T_{I_k}(z,\phi) & =  4kT(z) + T_{I_0}(z,\phi), \label{TimeIk}  
\end{align}
and likewise for the other types of orbits. 

\medskip

To build a global picture of the $I_k$ solution branch as $z$ increases
away from $z=0$ we
need to start by recalling what happens with the $I$ branch (i.e.: with
the case $k=0$). This was studied in sections \ref{sec:phaseplane*} and 
\ref{subsec:typeI}, and illustrated in figures~\ref{figtypeIend}
and \ref{localbif}: the $I$ branch of solutions collapses to a single point and disappears when $z=\sqrt{\phi^\star}$. 

\medskip

In the case of $I_k$ with $k\geqs 1$ an entirely different behaviour 
takes place: the $I_k$ branch of solutions exists for all $z<z^\star := 1+\cos 2\phi$,
where $z^\star$ is the value of $z$ of the point on the homoclinic orbit with
$x= -\phi$, and $T_{I_k}(z)\to +\infty$ when $z\to z^\star.$ In fact we have the two
situations we now describe: 

\medskip

Suppose $\phi\in\bigl(\frac{\pi}{4}, \frac{\pi}{2}\bigr)$, then 
$\cos 2\phi <0$ and thus $z^\star = 1+\cos 2\phi < 1-\cos 2\phi = (\phi^\star)^2$
which implies that the point where the branch of solutions of type $I$
collapse, $P^\star = (-\phi,\phi^\star)$ in Figure~\ref{figtypeIend}, 
is above the point $(-\phi,z^\star)$ at the homoclinic. This implies that,
as $z\uparrow z^\star$, the orbits of type $I_k$ converge to the 
homoclinic net constituted by the point 
$\bigl(-\frac{\pi}{2},0\bigr)\equiv \bigl(\frac{\pi}{2},0\bigr)$
and the two homoclinic orbits, as illustrated in Figure~\ref{homoclinicnet}.
The time spent will, obviously, converge to $+\infty$.

%
%
%
\begin{figure}[h]
	\psfragfig*[scale=0.60]{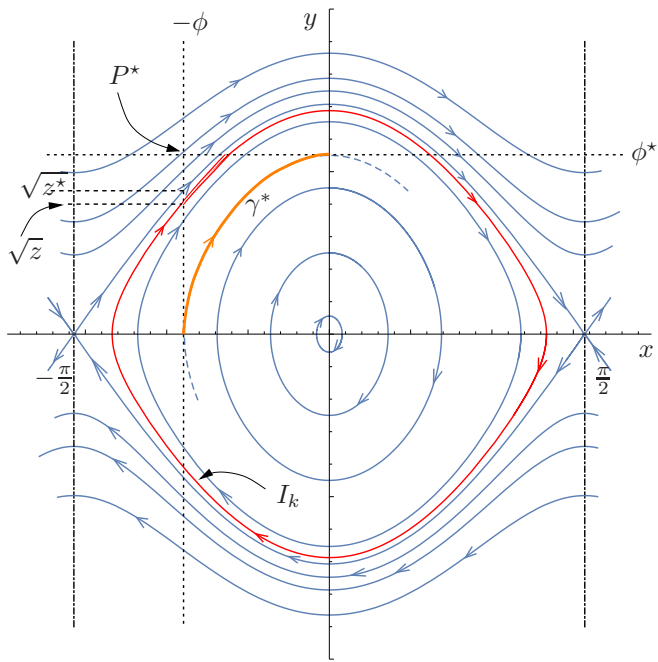}{%
	\psfrag{P}{$P^\star$}
	\psfrag{x}{$x$}	
	\psfrag{y}{$y$}
	\psfrag{z}{$\sqrt{z}$}	
	\psfrag{z*}{$\sqrt{z^\star}$}
	\psfrag{fi}{$-\phi$}
	\psfrag{I}{$I_k$}
	\psfrag{g}{$\gamma^*$}
	\psfrag{fi*}{$\phi^\star$}
	\psfrag{pi/2}{$\frac{\pi}{2}$}
	\psfrag{-pi/2}{$-\frac{\pi}{2}$}
}
	\caption{An orbit of type $I_k$, with $k=1$, and with $z$ close to $z^\star$.
		The case shown corresponds to $\phi\in\bigl(\frac{\pi}{4}, \frac{\pi}{2}\bigr)$, 
		in which case $P^\star$ is outside the homoclinic net.}\label{homoclinicnet}
\end{figure}
%
%
%

Assume now $\phi\in\bigl(0,\frac{\pi}{4}\bigr)$. Then 
$\cos 2\phi >0$ and thus $z^\star = 1+\cos 2\phi > 1-\cos 2\phi = (\phi^\star)^2$
which implies that the point $P^\star$ 
is now located inside the homoclinic net, as illustrated in Figure~\ref{insidehomoclinicnet}.  It is clear from this figure that 
the time spent by these type of orbits $I_k$ is given by 
\begin{align}
T_{I_k}(z,\phi) & =  4kT(z) - T_{I_0}(z,\phi). \label{TimeIkbis}  
\end{align} 

%
%
%
\begin{figure}[h]
\psfragfig*[scale=0.60]{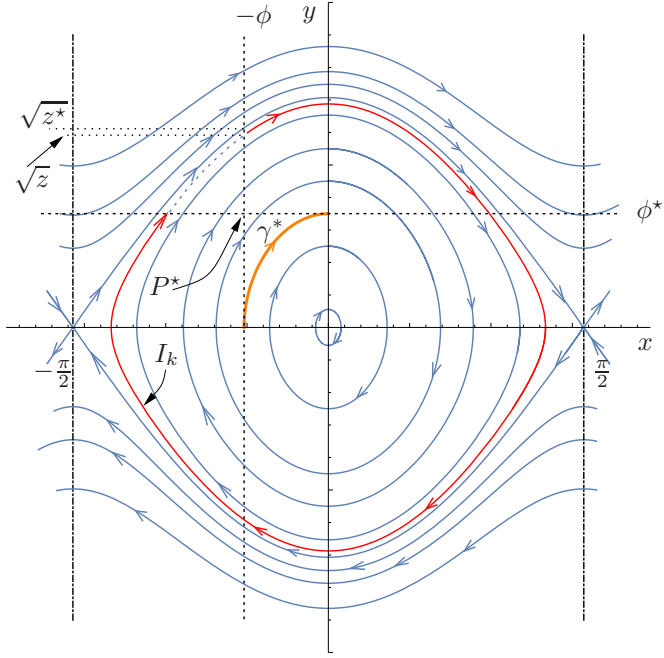}{%
	\psfrag{P}{$P^\star$}
	\psfrag{x}{$x$}
	\psfrag{y}{$y$}
	\psfrag{z}{$\sqrt{z}$}
	\psfrag{z*}{$\sqrt{z^\star}$}
	\psfrag{fi}{$-\phi$}
	\psfrag{I}{$I_k$}
	\psfrag{g}{$\gamma^*$}
	\psfrag{fi*}{$\phi^\star$}
	\psfrag{pi/2}{$\frac{\pi}{2}$}
	\psfrag{-pi/2}{$-\frac{\pi}{2}$}
}
	\caption{An orbit of type $I_k$, with $k=1$, and with $z$ close to $z^\star$.
		The case shown corresponds to $\phi\in\bigl(0,\frac{\pi}{4}\bigr)$, in which case $P^\star$ is inside the homoclinic net.}\label{insidehomoclinicnet}
\end{figure}
%
%
%


From Proposition~\ref{proptimemaps}, section~\ref{subsec:typeI}, 
and from $V(-\alpha, 0)=V(-\phi, \sqrt{z})$, which results in the relation $\alpha = \frac{1}{2}\arccos(\cos 2\phi -z)$, it is easy to conclude that orbits whose time map 
is given by  \eqref{TimeIkbis} satisfy 
\begin{align}
\frac{\partial T_{I_k}}{\partial z}(z,\phi) & > 0 \label{IncreasingTimeIk}  
\end{align}
and hence the corresponding solution branch in the bifurcation diagram is the graph of
a monotonic increasing function converging to the horizontal asymptote $z= z^\star$
as the time $L$ converges to $+\infty$.

\medskip

In the other hand, for orbits $I_k$ that do not enclose the point $P^\star$
(i.e., those illustrated in Figure~\ref{homoclinicnet}), the computation of 
$\frac{\partial T_{I_k}}{\partial z}(z,\phi)$ faces exactly the same problems as
we confronted in section~\ref{subsec:globalB} for the global behaviour of 
solution branch $B$.

\medskip

Actually, we can repeat the arguments in section~\ref{subsec:localB} to conclude that
$\frac{\partial T_{I_k}}{\partial z}(z,\phi)\to -\infty$ as $z\to 0$. 
Since, as we noted above, $T_{I_k}(z,\phi)\to +\infty$ as $z\to z^\star$,
we conclude that the $T_{I_k}(z,\phi)$ must have at least one minimum
for some\footnote{Observe that, due to \eqref{IncreasingTimeIk}, if
	the interval $(\phi^\star, z^\star)$ is not empty the piece of the solution branch 
	$I_k$ with $z$ in this interval is monotonic and so the minima of the $I_k$ branch
	must necessarily correspond to values $z< \phi^\star$.} 
$z\in (0, \phi^\star)$. A numerical study entirely analogous the one presented
in section~\ref{subsec:globalB} allows us to believe that Conjecture~\ref{C1}
is also valid for the branches $I_k$ with $k\geqs 1$. 

\medskip

From the fact that the time maps for the branches $B_k$ with $k\geqs 1$ are, like in
\eqref{TimeIk}, obtained from the one of branch $B$ by adding $4kT(z)$, the conclusion
we reached for the branches $I_k$ is repeated for the $B_k$s.

\medskip

Thus, from the discussion above and assuming Conjecture~\ref{C1}
holds true for the branches $I_k$ and $B_k$ with $k\geqs 1$, the
bifurcation scheme for the solution branches with $k\geqs 1$ is 
shown in Figure~\ref{globalIkbif}.

%
%
%
\begin{figure}[h]
	\psfragfig*[scale=0.27]{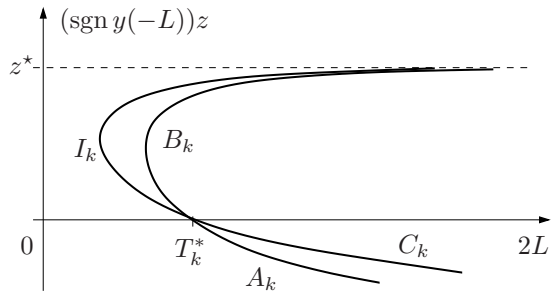}{%
	\psfrag{0}{$0$}
	\psfrag{L}{$2L$}
	\psfrag{y}{$(\sgn y(-L))z$}
	\psfrag{z}{$z^\star$}
	\psfrag{A}{$A_k$}
	\psfrag{B}{$B_k$}
	\psfrag{C}{$C_k$}
	\psfrag{I}{$I_k$}
	\psfrag{T0}{$T^*_k$}
}
	\caption{Bifurcating branches $I_k, A_k, B_k$ and $C_k$,
		with $k\geqs 1$, assuming the behaviour of the branch $B_0$
		stated in Conjecture~\ref{C1} is also valid for branches $I_k$ and $B_k$. 
		Here $T^*_k := 4kT(\phi).$}\label{globalIkbif}
\end{figure}
%
%
%

\medskip

Additional to these orbits, there are another two families of orbits, 
that we denote by $D_k$ and $D_k'$,
which correspond to orbits analogous to $B$ and $B'$ but are 
located above the homoclinic orbit
with positive $y$-component, and turn around the cylindrical phase space $k$ times (see
example in Figure~\ref{figDk}).

%
%
%
\begin{figure}[h]
	\psfragfig*[scale=0.60]{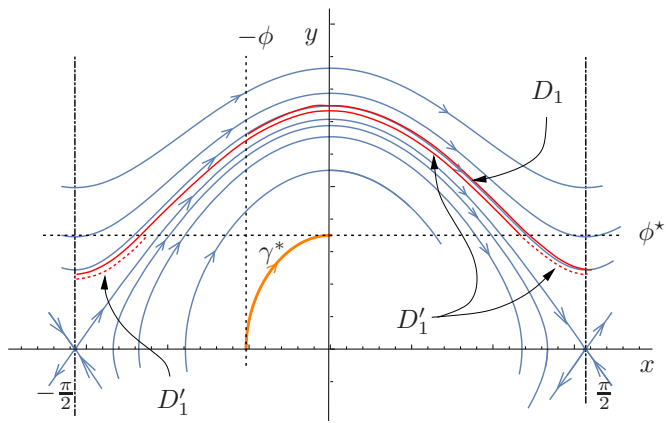}{%
	\psfrag{x}{$x$}	
	\psfrag{y}{$y$}
	\psfrag{fi}{$-\phi$}
	\psfrag{D}{$D_1$}
	\psfrag{D'}{$D_1'$}
	\psfrag{g}{$\gamma^*$}
	\psfrag{fi*}{$\phi^\star$}
	\psfrag{pi/2}{$\frac{\pi}{2}$}
	\psfrag{-pi/2}{$-\frac{\pi}{2}$}
}
	\caption{Orbits of type $D_k$ and $D_k'$, with $k=1$. The orbit $D_1'$ is
		represented by the full line (dashed and undashed); $D_1$ is represented by the undashed
		part only. In this plot each time the orbit $D_1$ or $D_1'$ passes through a same
		point the line is slightly deflected in order to facilitate the reading.}\label{figDk}
\end{figure}
%
%
%

The amount of time spent by these orbits are given by 
\begin{align}
T_{D_k}(z,\phi) & := 2kT_1\bigl(z,\tfrac{\pi}{2}\bigr)+ T_1(z,\phi)+T_1(z,\bar{x}^\star(z)) \label{TDkmap}\\
T_{D_k'}(z,\phi) & := 2(k+1)T_1\bigl(z,\tfrac{\pi}{2}\bigr)+ T_1(z,\phi)-T_1(z,\bar{x}^\star(z)), \label{TD'kmap}
\end{align}

\noindent
where, for each $\phi\in (0,\tfrac{\pi}{2})$, $z$ varies from $z^\star:=1+\cos 2\phi$, the value 
of ordinate square of the point of intersection of
$x=-\phi$ and the homoclinic orbit, obtained from $V(-\tfrac{\pi}{2},0)=V(-\phi,\sqrt{z^\star})$,
and $2$, the biggest value of $z$ for which the orbit intersects the line $y=\phi^\star$.

\medskip

We can see directly from Figure~\ref{figDk} that when $z=2$ the end points of $D_k$ and $D_k'$
are the same (and coincide with the point $(\tfrac{\pi}{2},\phi^\star)$.) This can also be obtained
from \eqref{TDkmap}-\eqref{TD'kmap}: 
\begin{eqnarray}
T_{D_k'}(z,\phi) & = & 2(k+1)T_1\bigl(z,\tfrac{\pi}{2}\bigr)+ T_1(z,\phi)-T_1(z,\bar{x}^\star(z))\nonumber \\
& = & T_{D_k}(z,\phi) + 2\bigl(T_1\bigl(z,\tfrac{\pi}{2}\bigr)-T_1(z,\bar{x}^\star(z))\bigr)\nonumber \\
& \geqs & T_{D_k}(z,\phi),	\label{Dk'>Dk}
\end{eqnarray}

\noindent
since, from its definition, $T_1(z,\cdot)$ is monotonic increasing, 
and, by \eqref{xstarz}, $\bar{x}^\star(z)\leqs\frac{\pi}{2}$ with equality only 
when $z=2$, because $\bar{x}^\star(2) = \frac{1}{2}\arccos(-1) = \frac{\pi}{2}.$

\medskip

From \eqref{Dk'>Dk} 
we conclude that, for each $k$, the branches $D_k$ and $D_k'$ exist
for all $z\in (z^\star, 2)$ and, in a bifurcation diagram with 
time as the bifurcation variable and $z$ as the dependent variable, 
the $D_k$ branch will always be to the left of the $D_k'$, 
except at one single point, with
ordinate $z=2$, where they coincide.

\medskip

To study the monotonicity of the branches $D_k$ and $D_k'$ we again use the time maps
\eqref{TDkmap} and \eqref{TD'kmap}, respectively. The last is easier: differentiating
\eqref{TD'kmap} with respect to $z$ we obtain, after some algebraic manipulations,
\begin{eqnarray}
\frac{\partial T_{D_k'}}{\partial z}(z,\phi) & = &
- (k+\tfrac{1}{2}) \int_0^{\tfrac{\pi}{2}}(z-\cos 2\phi+\cos 2x)^{-3/2}dx  \nonumber \\
& & -\, \tfrac{1}{2} \int_0^{\phi}(z-\cos 2\phi+\cos 2x)^{-3/2}dx   \nonumber \\
& & -\, \tfrac{1}{2} \int_{\bx(z)}^{\tfrac{\pi}{2}}(z-\cos 2\phi+\cos 2x)^{-3/2}dx \nonumber \\ 
& & -\,
\tfrac{1}{2}z^{-\frac{1}{2}}(2-z)^{-\frac{1}{2}}(1 - \cos 2\phi)^{-\frac{1}{2}}\nonumber \\
& < & 0,
\end{eqnarray}

\noindent
and $\frac{\partial T_{D_k'}}{\partial z} \to -\infty$ when $z\to z^\star$ and when $z\to 2$.
 
\medskip

The same computation for the branch $D_k$ runs into the difficulties already encountered 
before when studying the branches $B_k$ (with $k\geqs 0$) and $I_k$ (with $k\geqs 1$). The
expression for the derivative of the time map is
\begin{eqnarray}
\frac{\partial T_{D_k}}{\partial z}(z,\phi) & = &
- (k+1) \int_0^{\tfrac{\pi}{2}}(z-\cos 2\phi+\cos 2x)^{-3/2}dx  \nonumber \\
& & -\, \tfrac{1}{2} \int_0^{\phi}(z-\cos 2\phi+\cos 2x)^{-3/2}dx   \nonumber \\
& & -\, \tfrac{1}{2} \int_0^{\bx(z)}(z-\cos 2\phi+\cos 2x)^{-3/2}dx \nonumber \\ 
& & +\,
\tfrac{1}{2}z^{-\frac{1}{2}}(2-z)^{-\frac{1}{2}}(1 - \cos 2\phi)^{-\frac{1}{2}},
\end{eqnarray}
and $\frac{\partial T_{D_k}}{\partial z} \to -\infty$ when $z\to z^\star$, and 
$\frac{\partial T_{D_k}}{\partial z} \to +\infty$ when $z\to 2$.
The existence of a unique minimum of $T_{D_k}(\cdot,\phi)$ can be checked 
numerically like was done in the case of the $B$ branch in section~\ref{subsec:globalB}, but,
as there, the rigorous proof eludes us at present. All numerical evidence points
to the validity of Conjecture~\ref{C1} also for the function $z\mapsto T_{D_k}(z,\phi),$
and assuming this the situation with branches $D_k$ and $D_k'$ is illustrated
in Figure~\ref{globalDkbif}.

%
%
%
\begin{figure}[h]
	\psfragfig*[scale=0.35]{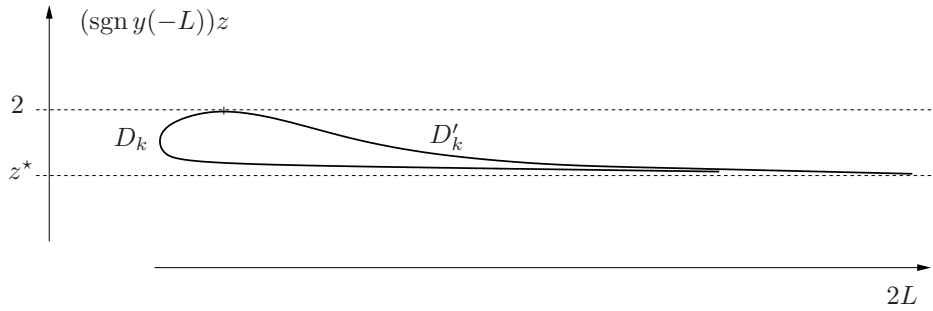}{%
	\psfrag{2}{$2$}
	\psfrag{L}{$2L$}
	\psfrag{y}{$(\sgn y(-L))z$}
	\psfrag{hp}{$z^\star$}
	\psfrag{D1}{$D_k$}
	\psfrag{D1'}{$D_k'$}
	\psfrag{fi*}{$\phi^\star$}
}
	\caption{Solution branches $D_k$ and $D_k'$,
		with $k\geqs 1$, assuming that the behaviour of branch $B$
		stated in Conjecture~\ref{C1} is also valid for branches $D_k$.}\label{globalDkbif}
\end{figure}
%
%
%

\section{Conclusion}\label{sec:conclusion}

Gathering the results and discussions from the previous sections we can construct the following 
schematic bifurcation diagram for solutions of \eqref{originalsystem}--\eqref{originalbc}.
%
%
%
\begin{figure}[h]
	\psfragfig*[scale=0.24]{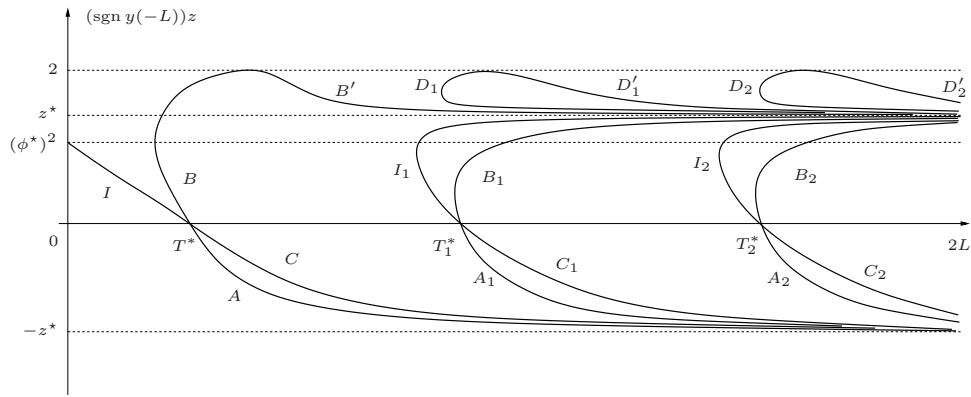}{%
	\psfrag{L}{\tiny{$2L$}}	
	\psfrag{y}{\tiny{$(\sgn y(-L))z$}}
	\psfrag{star}{\tiny{$(\phi^\star)^2$}}
	\psfrag{hp}{\tiny{$z^\star$}}
	\psfrag{-hp}{\tiny{$-z^\star$}}
	\psfrag{0}{\tiny{$0$}}
	\psfrag{2}{\tiny{$2$}}
	\psfrag{A0}{\tiny{$A$}}
	\psfrag{A1}{\tiny{$A_1$}}
	\psfrag{A2}{\tiny{$A_2$}}
	\psfrag{B0}{\tiny{$B$}}
	\psfrag{B'}{\tiny{$B'$}}
	\psfrag{B1}{\tiny{$B_1$}}
	\psfrag{B2}{\tiny{$B_2$}}
	\psfrag{C0}{\tiny{$C$}}
	\psfrag{C1}{\tiny{$C_1$}}
	\psfrag{C2}{\tiny{$C_2$}}
	\psfrag{D1}{\tiny{$D_1$}}
	\psfrag{D1'}{\tiny{$D_1'$}}
	\psfrag{D2}{\tiny{$D_2$}}
	\psfrag{D2'}{\tiny{$D_2'$}}
	\psfrag{I0}{\tiny{$I$}}
	\psfrag{I1}{\tiny{$I_1$}}
	\psfrag{I2}{\tiny{$I_2$}}
	\psfrag{T0}{\tiny{$T^*$}}
	\psfrag{T1}{\tiny{$T^*_1$}}
	\psfrag{T2}{\tiny{$T^*_2$}}
	\psfrag{g}{\tiny{$\gamma^*$}}
	\psfrag{fi*}{\tiny{$\phi^\star$}}
	\psfrag{pi/2}{\tiny{$\frac{\pi}{2}$}}
	\psfrag{-pi/2}{\tiny{$-\frac{\pi}{2}$}}
	\centering
}
	\caption{Schematic bifurcation diagram for system \eqref{originalsystem}--\eqref{originalbc} with $\phi\in \bigl(0,\tfrac{\pi}{4}\bigr)$, assuming validity of Conjecture~\ref{C1} for the
		relevant branches, as explained in the text. For the notation
		used in this figure see the text.
	}\label{fbifurcations}
\end{figure}
%
%
%

\section*{Acknowledgements}

\noindent
FPdC and JTP were partially funded by FCT/Portugal through project RD0447/ CAMGSD/2015.
FPdC acknowledges financial support provided by the University of Strathclyde
David Anderson Research Professorship.
Parts of sections 3, 4.1, and 4.2 first appeared in the dissertation submitted by KX 
to the National University of Laos, Vientiane, Laos, in July 2017,
as part of the requirements for the degree of MSc in Applied Mathematics.

\end{document}